\documentclass[11pt]{article}
\usepackage{amsthm}
\usepackage{amsmath}
\usepackage{amssymb}

\usepackage{framed, fullpage}
\usepackage{graphicx}

\usepackage[backref,colorlinks,citecolor=blue,bookmarks=true]{hyperref}
\usepackage{subfig}

\newtheorem{theo}{Theorem}
\newtheorem{coro}{Corollary}

\newtheorem{lemma}{Lemma}[section]
\newtheorem{definition}[lemma]{Definition}

\newtheorem{claim}[lemma]{Claim}

\newtheorem{obs}[lemma]{Observation}

\newcommand{\x}{\mathbf{x}}

\renewcommand{\Pr}{\mathop{\bf Pr\/}}

\newcommand{\Z}{\mathbb{Z}}

\renewcommand{\P}{\mathcal{P}}
\newcommand{\F}{\mathcal{F}}

\newcommand{\size}[1]{\left\lvert #1 \right\rvert}

\title{Ramsey Theory, Integer Partitions and a New Proof of the Erd\H{o}s-Szekeres Theorem}
\author{
Guy Moshkovitz\thanks{School of Mathematics, Tel-Aviv University, Tel-Aviv, Israel 69978.  Email: {\tt guymosko@tau.ac.il}. Supported in part by ISF grant 224/11.}
\and Asaf Shapira\thanks{School of Mathematics, Tel-Aviv University, Tel-Aviv, Israel 69978, and Schools of Mathematics and Computer Science, Georgia Institute of Technology, Atlanta, GA 30332. Email: {\tt asafico@tau.ac.il}. Supported in part by NSF Grant DMS-0901355, ISF Grant 224/11 and a Marie-Curie CIG Grant 303320.}
}

\begin{document}

\date{}
\maketitle

\begin{abstract}

%
Let $H$ be a $k$-uniform hypergraph whose vertices are the integers $1,\ldots,N$. We say that $H$ contains a monotone path of length $n$ if there are $x_1 < x_2 < \cdots < x_{n+k-1}$ so that $H$ contains all $n$ edges of the form $\{x_i,x_{i+1},\ldots,x_{i+k-1}\}$.
Let $N_k(q,n)$ be the smallest integer $N$ so that every $q$-coloring of the edges of the complete $k$-uniform hypergraph on $N$ vertices contains a monochromatic monotone path of
length $n$.
While the study of $N_k(q,n)$ for specific values of $k$ and $q$ goes back (implicitly) to the seminal 1935 paper of Erd\H{o}s and Szekeres, the problem of bounding $N_k(q,n)$ for arbitrary $k$ and $q$ was studied by Fox, Pach, Sudakov and Suk.

Our main contribution here is a novel approach for bounding the Ramsey-type numbers $N_k(q,n)$, based on establishing a surprisingly tight connection between them and the enumerative problem of counting high-dimensional integer partitions. Some of the concrete results we obtain using this approach are the following:

\begin{itemize}
\item We show that for every fixed $q$ we have $N_3(q,n)=2^{\Theta(n^{q-1})}$, thus resolving an open problem raised by Fox et al. 

\item We show that for every $k \geq 3$,  $N_k(2,n)=2^{\cdot^{\cdot^{2^{(2-o(1))n}}}}$
where the height of the tower is $k-2$, thus resolving an open problem raised by Eli\'{a}\v{s} and Matou\v{s}ek. 

\item We give a new pigeonhole proof of the Erd\H{o}s-Szekeres Theorem on cups-vs-caps, similar to Seidenberg's proof of the Erd\H{o}s-Szekeres Lemma on increasing/decreasing subsequences.
\end{itemize}

\end{abstract}


\section{Introduction} \label{sec:intro}

\subsection{Some historical background}

It would not be an exaggeration to state that modern Extremal Combinatorics, and Ramsey Theory in particular, stemmed from the seminal 1935 paper of Erd\H{o}s and Szekeres~\cite{ErdosSz35}. Besides establishing explicit bounds for graph and hypergraph Ramsey numbers, they also proved two of the most well-known results in Combinatorics, which have become known as the
{\em Erd\H{o}s-Szekeres Lemma} (ESL) and the {\em Erd\H{o}s-Szekeres Theorem} (EST).
Let $f(a,b)$ be the smallest integer so that every sequence of $f(a,b)$ distinct real numbers contains either an increasing sequence of length $a$ or a decreasing sequence of length $b$.
Then ESL states that
$$
f(n,n) \leq (n-1)^2+1\;.
$$
Let $g(a,b)$ be the smallest integer so that every set of $g(a,b)$ points in the plane in general position, all with distinct $x$-coordinates, contains either $a$ points $p_1,\ldots,p_{a}$
with increasing $x$-coordinate so that the slopes of the segments $(p_1,p_2),(p_2,p_3)\ldots,(p_{a-1},p_{a})$ are
increasing, or $b$ such points so that the slopes of these segments are decreasing.
Then EST states that
\begin{equation}\label{eq:EST}
g(n,n) \leq \binom{2n-4}{n-2}+1\;.
\end{equation}
We note that EST implies that for any integer $n$ there is an integer $N(n)$ so that every set of $N(n)$ points in general position in the plane contains $n$ points in convex position.
Specifically, it shows that $N(n) \leq \binom{2n-4}{n-2}+1$. The fact that $N(n)$ is finite was later labelled the ``Happy Ending Theorem''.

The original proof in~\cite{ErdosSz35} of ESL was based on establishing the recurrence relation
$f(n+1,n+1) \leq f(n,n)+2n-1$.
By now, there are several proofs of ESL.
In fact, Steele~\cite{Steele95} has collected $7$ of these proofs, and dubbed the following pigeonhole-type proof by Seidenberg~\cite{Seidenberg59} as ``the slickest and most systematic''.
Assign to each real number $x$ in the sequence two labels $x^+,x^-$ where we take $x^+$ to be the length of the longest increasing sequence ending at $x$, and $x^-$ to be the length of the longest decreasing sequence ending at $x$. Now, it is easy to see that for every pair of reals $x,y$ in the sequence we have $(x^+,x^-) \neq (y^+,y^-)$.
Hence, if there is neither an increasing nor a decreasing sequence of length $n$ then there can be no more than $(n-1)^2$ numbers in the sequence. The same idea shows that $f(a,b) \leq (a-1)(b-1)+1$, and it is easy to see that this bound is tight.

The original proof in~\cite{ErdosSz35} of EST was based on establishing the recurrence relation $g(a+1,b+1) \leq g(a,b+1)+g(a+1,b)-1$. To the best of our knowledge, this is the only known proof of this classic result.\footnote{We stress that here we are referring to bounding $g(a,b)$. The bound on $N(n)$ has been slightly improved (see~\cite{MorrisSo00} for a survey). Interestingly, all improvements rely
on clever applications of $g(a,b)$.} As part of our investigation here we will establish a new pigeonhole-type proof of EST, similar in spirit to Seidenberg's proof~\cite{Seidenberg59} of ESL we sketched in the previous paragraph.
We have to admit that we did not set out to try and find a new proof of EST. Our goal was actually to bound Ramsey numbers of certain generalizations of EST, and the new proof is just a byproduct.

\subsection{High-dimensional integer partitions}

The notion of integer partitions is without doubt the most well-studied notion in discrete mathematics, and goes back (at least) to Euler.
We will be very brief here and just define the notions that are relevant to the results of this paper (see~\cite{Andrews} for more background on this subject).
A decreasing sequence of nonnegative integers $a_1 \geq a_2 \geq \ldots$ will be called a {\em line partition}. One can visualize a line partition
as a 2-dimensional sequence of stacks of height $a_i$ each (essentially, a Young diagram, see Figure~\ref{fig:1partition}).
A matrix $A$ of nonnegative integers so that $A_{i,j} \geq A_{i+1,j}$ and $A_{i,j} \geq A_{i,j+1}$ for all possible $i,j$ will be called a {\em plane partition}. One can visualize a plane partition in 3-dimensions as a plane consisting of stacks, where at location $(i,j)$ we have a stack of height $A_{i,j}$ (see Figure~\ref{fig:2partition}). The notion of a plane partition was introduced by MacMahon in 1897~\cite{MacMahon97} as a $2$-dimensional analogue of integer partitions\footnote{Recall that if $n=a_1+\ldots+a_k$ then the standard way to write this partition is as a decreasing sequence $a_1 \geq a_2 \geq \ldots$, that is, what we call here a line partition.}
and has been extensively studied ever since.
More generally, one defines a $d$-dimensional partition as a $d$-dimensional (hyper)matrix $A$ of nonnegative integers so that the matrix is decreasing in each line, that is,
$A_{i_1,\ldots,i_t,\ldots,i_{d}} \geq A_{i_1,\ldots,i_t+1,\ldots,i_{d}}$
for every possible $i_1,\ldots,i_{d}$ and $1 \leq t \leq d$.

\begin{figure}
\centering
\subfloat[A line partition and the lattice path along its boundary.]
{\label{fig:1partition}\includegraphics[scale=0.5]{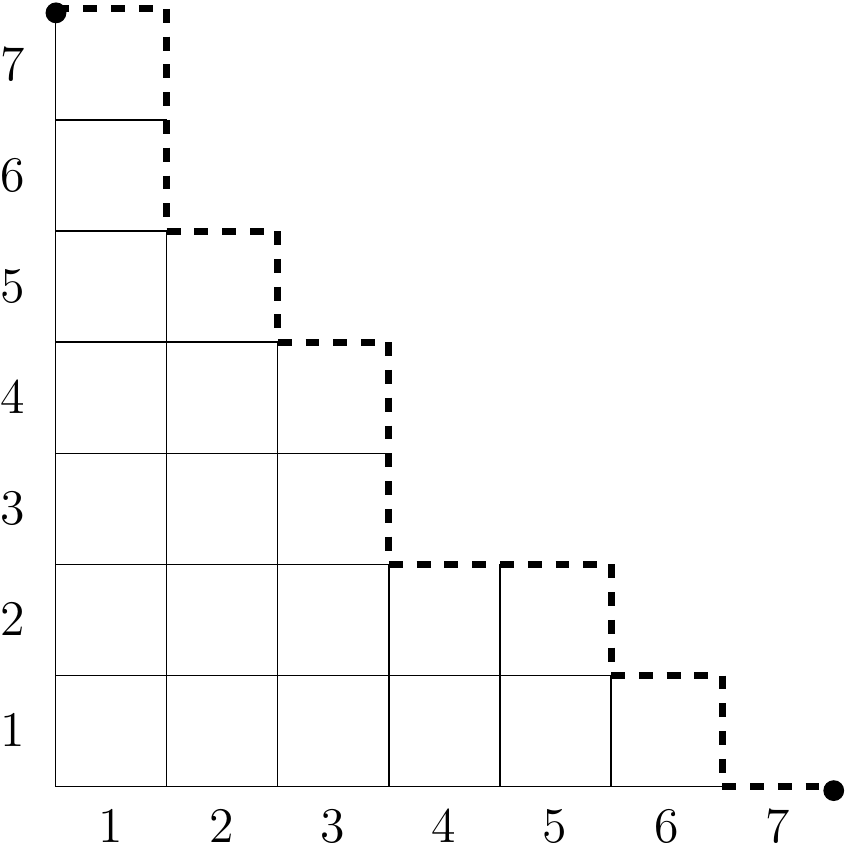}}
\qquad\qquad\qquad
\subfloat[A plane partition. Note that the stacks are ``flushed into the corner''.]
{\label{fig:2partition}\includegraphics[scale=0.5]{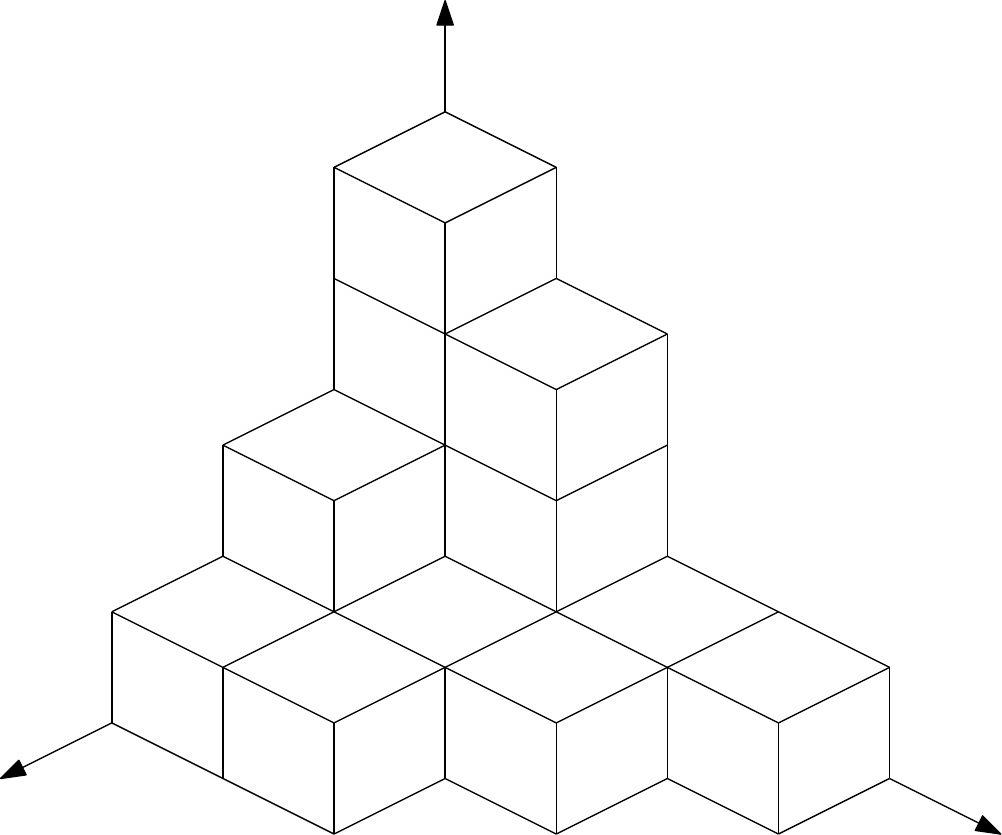}}
\caption{Partitions}
\end{figure}

Let $P_d(n)$ be the number of $n\times\cdots\times n$ $d$-dimensional partitions with entries from $\{0,1,\ldots,n\}$.
Note that when $d=1$ (that is, line partitions $n \geq a_1 \geq \ldots \ge a_n \geq 0$) we can think of such an integer partition as a lattice path in $\Z^2$ starting at $(0,n)$ and ending at $(n,0)$ where in each step the path moves either down or to the right (see Figure~\ref{fig:1partition}).
It is thus clear that
\begin{equation}\label{eq:dim1partition}
P_1(n)=\binom{2n}{n} \;.
\end{equation}
Computing $P_2(n)$ appears to be much harder. Luckily, a celebrated result of MacMahon~\cite{MacMahon16} states that

\begin{equation}\label{eq:dim2partition}
P_2(n)=\prod_{1 \leq i,j,k \leq n}\frac{i+j+k-1}{i+j+k-2}\;.
\end{equation}
We refer the reader to~\cite{Stanley89} for more background and references on the rich history and current research on plane partitions. We also recommend Chapter 5 of~\cite{Aigner} for an ingenious proof of~(\ref{eq:dim2partition}) using the Lindstr\"{o}m-Gessel-Viennot Lemma.

Unfortunately, there is no known closed formula for $P_d(n)$ even for $d=3$ (see, e.g.,~\cite{Bailey99}).
The same is true even for the more well-studied variant of the problem, which considers the number of partitions with a given sum of entries; in fact, even establishing a generating function for three-dimensional partitions, usually referred to as {\em solid partitions}, is an outstanding open problem in enumerative combinatorics which goes back to MacMahon. See~\cite{MustonenRa03} and~\cite{Stanley89} for some background on this open problem.
However, observe that each line in a $d$-dimensional partition is a line partition. Since a $d$-dimensional partition is composed of $n^{d-1}$ line partitions, we can derive from (\ref{eq:dim1partition}) the crude bound

\begin{equation}\label{eq:dimqpartition}
P_d(n) \leq \binom{2n}{n}^{n^{d-1}} \leq 2^{2n^{d}}\;.
\end{equation}

\subsection{Erd\H{o}s-Szekeres generalized}\label{subsec:results1}

One cannot but suspect that there is some abstract combinatorial phenomenon behind ESL and EST. Indeed, ESL is a special case of Dilworth's Theorem~\cite{Dilworth50} (or actually, its dual, which is due to Mirsky~\cite{Mirsky71}). As to EST, Chv\'{a}tal and Koml\'{o}s~\cite{Chvatal71} obtained a combinatorial lemma generalizing it in terms of paths in edge weighted tournaments. Very recently, Fox, Pach, Sudakov and Suk~\cite{FoxPaSuSu11} suggested the following elegant framework for studying such problems, which nicely puts both ESL and EST under a single roof.

Let $K^k_N$ denote the complete $k$-uniform hypergraph on a set of $N$ vertices, that is, the collection of subsets of size $k$ of the $N$ vertices. For our purposes here it will be useful to think of the vertices as being ordered and thus name them $1,\ldots,N$. For a sequence of vertices $x_1 < x_2 < \cdots < x_{n+k-1}$ we say that the edges
$$
\{x_1,\ldots,x_k\}, \{x_2,\ldots,x_{k+1}\},\ldots,\{x_{n},\ldots,x_{n+k-1}\}
$$
form a {\em monotone path}, and we refer to the number of edges as its {\em length} (so the path above is of length $n$).\footnote{We note that Fox et al.~\cite{FoxPaSuSu11} measured the length of a path using the number of vertices. As it turns out, in our proofs it will be much more natural to use the number of edges as the measure of length.}
Note that henceforth, whenever we will be talking about an edge $\{x,y,z\}$ we will implicitly assume that $x < y <z$.

Let $N_k(q,n)$ be the smallest integer $N$ so that every coloring of the edges of $K^k_N$ using $q$ colors contains a monochromatic monotone path of length $n$.
Recalling ESL, it is easy to see that
$f(n+1,n+1) \leq N_2(2,n)$ (notice $N_2(q,n)$ measures length with respect to edges),
and that the proof of ESL we sketched earlier implies $N_2(q,n) \leq n^q+1$. It is also easy to see that
\begin{equation}\label{eq:PathEST}
g(n+2,n+2) \leq N_3(2,n)\;,
\end{equation}
and that one can prove $N_3(2,n) \leq \binom{2n}{n}+1$ by applying the recursive argument (mentioned above) that was used by Erd\H{o}s and Szekeres~\cite{ErdosSz35} to prove EST~(\ref{eq:EST}). Hence, one can prove both ESL and EST in the framework of studying $N_k(q,n)$.
%

The question of bounding $N_3(q,n)$ for $q \geq 3$ was raised by Fox et al.~\cite{FoxPaSuSu11}, motivated (partially) by certain geometric generalizations of EST (see~\cite{FoxPaSuSu11} for the exact details). One of their main results was that
\begin{equation}\label{FPSS}
2^{(n/q)^{q-1}}\leq N_3(q,n) \leq n^{n^{q-1}} \;.
\end{equation}
The main problem they left open was whether the correct exponent $\log_2 N_3(q,n)$ is of order $n^{q-1}$.

\subsection{Our results}

Our main result in this paper establishes a surprisingly close connection between the problem of bounding the Ramsey numbers $N_3(q,n)$ defined above, and the problem of enumerating high-dimensional integer partitions we discussed in the previous subsection.

\begin{theo}\label{theo:main} For every $q \geq 2$ and $n \geq 2$ we have
$$
N_3(q,n)=P_{q-1}(n)+1 \;.
$$
\end{theo}

Recall that by~(\ref{eq:dimqpartition}) we have $P_{q-1}(n) \leq 2^{2n^{q-1}}$. Hence, Theorem~\ref{theo:main} resolves the problem of Fox et al.~\cite{FoxPaSuSu11} mentioned above by establishing that for every fixed $q$ we have $N_3(q,n)=2^{\Theta(n^{q-1})}$.
Also, we get via MacMahon's formula~(\ref{eq:dim2partition}) an exact bound for $N_3(3,n)$, and from the long history of unsuccessful attempts to precisely compute $P_3(n)$ we learn that it is probably hopeless to compute $N_3(4,n)$ exactly.
In fact, as we shall see later on in the paper, Theorem~\ref{theo:main} also implies that it is unlikely to expect a closed formula even for $N_3(q,2)$, namely the case of the two-edge path! (see Section~\ref{sec:conclude})

Notice that Theorem~\ref{theo:main}, together with the observations we have made in (\ref{eq:dim1partition}) and (\ref{eq:PathEST}), implies EST~(\ref{eq:EST}).
We note that our initial approach for resolving the problem raised in~\cite{FoxPaSuSu11} was to adapt the recursive approach of Erd\H{o}s and Szekeres for proving EST (\ref{eq:EST}) (as sketched in the previous section) to the general setting of $q \geq 3$. It appears that this approach cannot be generalized, mainly because attempts to come up with a recursion for $P_{q-1}(n)$ have failed.
So in some sense, our new proof of EST came out of the need to find a proof that can be generalized to more than $2$ colors.

To prove the upper bound in Theorem~\ref{theo:main} we will use a pigeonhole type argument, similar to Seidenberg's proof~\cite{Seidenberg59} of ESL which we sketched above. To this end we will map every
vertex of the hypergraph to a $(q-1)$-dimensional partition and argue that this mapping must be injective. To prove the lower bound, we will show a surprising way by which one can think
of $(q-1)$-dimensional partitions as vertices of a hypergraph, and then use certain relations between these partitions/vertices in order to define an {\em explicit} $q$-coloring of the complete $3$-uniform hypergraph without long monochromatic monotone paths.

Note that the bounds on $N_3(q,n)$ we get from Theorem~\ref{theo:main} by combining (\ref{eq:dimqpartition}) and the lower bound in (\ref{FPSS}) are $2^{(n/q)^{q-1}} \leq N_3(q,n) \leq 2^{2n^{q-1}}$.
It is thus natural to ask if one can use the precise description of $N_3(q,n)$ of Theorem~\ref{theo:main} in order to tighten the dependence on $q$ in the exponent of $N_3(q,n)$.
To this end we first prove the following.

\begin{theo}\label{theo:lower} For every $d \geq 1$ and $n \geq 1$ we have
$$ 
P_d(n) \geq 2^{\frac23n^{d}/\sqrt{d+1}}\;.
$$ 
\end{theo}

Observe that from the above (and Theorem~\ref{theo:main}) we immediately get an exponential improvement over the lower bound of Fox et al.~\cite{FoxPaSuSu11} (stated in (\ref{FPSS})) in terms of the dependence of $N_3(q,n)$ on $q$. The following corollary thus summarizes our bounds for $N_3(q,n)$.

\begin{coro}\label{coro:bound}
For every $q \geq 2$ and $n \geq 2$ we have
$$
2^{\frac23 n^{q-1}/\sqrt{q}} \leq N_3(q,n) \leq 2^{2n^{q-1}} \;.
$$
\end{coro}

As we discuss in the concluding remarks, it seems reasonable to conjecture that the lower bound gives the correct exponent, and
that $N_3(q,n)=2^{\Theta(n^{q-1}/\sqrt{q})}$.

\subsection{Higher uniformity}\label{subsection:results2}

Given the characterization of $N_3(q,n)$ in terms of enumerating integer partitions, it is natural to ask if a similar characterization can be proved for $N_k(q,n)$ for arbitrary $k \geq 3$. As we show in Section~\ref{sec:kgraphs} the answer is positive, but since the objects involved in this characterization are (slightly) complicated to define, we refrain from stating the results in this section and refer the reader to the statement of Theorem~\ref{theo:Nk2exact} in Section~\ref{sec:kgraphs}. Let us instead describe some immediate corollaries of this characterization.

Let $t_k(x)$ be a tower of exponents of height $k-1$ with $x$ at the top. So $t_1(x)=x$ and $t_3(x)=2^{2^{x}}$.
Since we know that $N_3(q,n) \leq 2^{O(n^{q-1})}$, it seems natural to suspect that $N_k(q,n) \leq t_{k-1}(O(n^{q-1}))$.
Indeed, Fox et al.~\cite{FoxPaSuSu11} show how to convert a bound of the form $N_3(q,n) \leq 2^{cn^{q-1}}$ (which we obtain in Theorem~\ref{theo:main}) into the more general bound $N_k(q,n) \leq t_{k-1}(c'n^{q-1})$ (where $c,c'$ are absolute constants). Here, we apply the main idea behind the proof of Theorem~\ref{theo:main} to give a very short and direct proof of the following bound, improving the one from~\cite{FoxPaSuSu11}, which relates $N_k(q,n)$ to $N_3(q,n)$.
\begin{theo}\label{theo:transform} For every $k \geq 3$, $q\geq 2$, and $n\geq 2$ we have
$$
N_k(q,n) \leq t_{k-2}(N_3(q,n)) \;.
$$
\end{theo}

This implies, by the upper bound in Corollary~\ref{coro:bound}, that $N_k(q,n)\leq t_{k-1}(2n^{q-1})$ for $k\geq 3$, as desired.
We note that using our enumerative characterization of $N_k(q,n)$ we can actually get a better upper bound than the one stated above; see the discussion in Section~\ref{sec:conclude}.

While we can prove Theorem~\ref{theo:transform} directly, and without using our characterization for $N_k(q,n)$, it appears that to prove a matching lower bound does require this enumerative characterization. Specifically, we have the following.

\begin{theo}\label{theo:transformlower} There is an absolute constant $n_0$, so that for every $k\geq 3$, $q\geq 2$ and $n \geq n_0$, we have
$$
N_k(q,n)\geq t_{k-2}(N_3(q,n)/3n^q) \;.
$$
\end{theo}

The above result improves upon a lower bound of~\cite{FoxPaSuSu11}, but more importantly, matches (up to lower order terms) the upper bound of Theorem~\ref{theo:transform}.
One application of the above tight bounds is the following. Eli\'{a}\v{s} and Matou\v{s}ek~\cite{EliasMa11} have recently introduced another framework for studying generalizations of both ESL and EST in terms of the $k^{th}$ derivative of the function passing through a set of points. Motivated by the relation between their framework and the one introduced in~\cite{FoxPaSuSu11}, they asked
if for every $k\geq 3$ one has $N_k(2,n)=t_{k-1}(\Theta(n))$.
By combining Theorems~\ref{theo:transform}, Theorem~\ref{theo:transformlower}, and our bounds on $N_3(q,n)$ in Corollary~\ref{coro:bound}, we in particular get the following sharp (and positive) answer.

\begin{coro}\label{coro:matousek} For every $k \geq 3$ we have
$$
N_k(2,n)=t_{k-1}((2-o(1))n)\;,
$$
where the $o(1)$ term goes to $0$ as $n \rightarrow \infty$.
\end{coro}

In fact, we may deduce the following, summarizing our bounds for general $q$.

\begin{coro}\label{coro:hyper} For every $k \geq 3$, $q \geq 2$, and sufficiently large $n$
we have
$$
t_{k-1}(n^{q-1}/2\sqrt{q})  \leq N_k(q,n) \leq t_{k-1}(2n^{q-1}) \;.
$$
\end{coro}

\paragraph*{Organization:}
The rest of the paper is organized as follows. In Section~\ref{sec:3graphs} we mainly focus on $3$-uniform hypergraphs. We first give a new proof of EST by showing that $N_3(2,n) \leq P_1(n)+1$. We will then move on to prove the more general bound $N_3(q,n) \leq P_{q-1}(n)+1$. The proof of the general bound will turn out to be almost identical to the proof of the case $q=2$. Next we will prove the lower bound of Theorem~\ref{theo:main}, thus completing the characterization of $N_3(q,n)$. Since the proof of Theorem~\ref{theo:transform} (giving an upper bound on our Ramsey number for $k$-uniform hypergraphs) is so similar to the proof of Theorem~\ref{theo:main}, we will also give this proof in Section~\ref{sec:3graphs}. We will end Section~\ref{sec:3graphs} with the proof of Theorem~\ref{theo:lower}.

In Section~\ref{sec:kgraphs} we consider $k$-uniform hypergraphs.
We will start with proving a theorem analogous to Theorem~\ref{theo:main}, giving a characterization of $N_k(q,n)$ in terms of enumerating higher-order variants of $P_{q-1}(n)$. We will then show how one can derive the lower bound in Theorem~\ref{theo:transformlower} from this characterization.
In Section~\ref{sec:conclude} we give some concluding remarks and open problems; among other things, we discuss how the problem of estimating the Ramsey-type numbers $N_3(q,k)$, or equivalently, estimating the number of $d$-dimensional integer partitions $P_d(n)$, naturally leads to a problem of estimating the number of independent sets in graphs, a well-studied problem in enumerative combinatorics.

\section{New Bounds for $3$-Uniform Hypergraphs}\label{sec:3graphs}
We write $[n]$ for $\{1,\ldots,n\}$.
For $x,y\in[n]^d$ we denote $x\preceq y$ when $x_i\le y_i$ for all $1\le i\le d$.
A set $S\subseteq[n]^d$ is a \emph{down-set} if $s\in S$ implies $x\in S$ for all $x\preceq s$.
We will frequently use the following simple observation
stating that any down-set can be viewed as a $d-1$-dimensional partition.
This is best explained by Figures~\ref{fig:1partition} and~\ref{fig:2partition}, but we include the formal proof for completeness.

\begin{obs}\label{obs:down-sets}
The number of down-sets $S\subseteq[n]^d$ is $P_{d-1}(n)$.
\end{obs}
\begin{proof}
We injectively map every down-set $S\subseteq[n]^d$ to a $(d-1)$-dimensional integer partition as follows.
For every $1\le i_1,\ldots,i_{d-1}\le n$ let
$A_{i_1,\ldots,i_{d-1}}=\max\{s : (i_1,\ldots,i_{d-1},s)\in S\}$ (where by convention, maximum over an empty set is $0$).
Then clearly $0\leq A_{i_1,\ldots,i_{d-1}} \leq n$ and, since $S$ is a down-set, it easily follows that
$A_{i_1,\ldots,i_t,\ldots,i_{d-1}} \geq A_{i_1,\ldots,i_t+1,\ldots,i_{d-1}}$
for every possible $1 \leq t \leq d-1$.
In other words, the (hyper)matrix $A$ is a $(d-1)$-dimensional $n\times\cdots\times n$ partition with entries from $\{0,1,\ldots,n\}$.
Furthermore, one can easily verify that this defines a bijection.
\end{proof}

\subsection{A new proof of the Erd\H{o}s-Szekeres Theorem}\label{subsec:newEST}

\begin{proof}
Fix a black/white coloring of the edges of $K^3_N$ that has no monochromatic monotone path of length $n$.
We need to show that $N \leq P_1(n)$.
For every pair of vertices $u<v$ 
denote $C(uv):=(1+n_b,1+n_w)$ where $n_b$ is the length (i.e., number of edges) of the longest black monotone path ending with $\{u,v\}$, and $n_w$ is defined in a similar way only with respect to white monotone paths. Notice that $C(uv) \in [n]^2$.
%
%
Define
$$D(v)=\{x \in [n]^2 : x \preceq C(uv) \mbox{ for some }u<v \} \;,$$
and note that $D(v)$ is (by definition) a down-set in $[n]^2$. It thus follows from Observation~\ref{obs:down-sets} (and the pigeonhole principle) that it is enough to show that $D(u) \neq D(v)$ for every pair of vertices.
So suppose to the contrary that $u < v$ and $D(u) = D(v)$.
By definition, $C(uv)\in D(v)$, and thus $C(uv)\in D(u)$.
Hence, (again, by definition) there is a vertex $t<u$ such that $C(uv)\preceq C(tu)$.
However, if the edge $\{t,u,v\}$ is colored black then we can extend the longest black monotone path ending at $\{t,u\}$ to a longer one ending at $\{u,v\}$, and similarly if $\{t,u,v\}$ is colored white. In either case we have $C(uv) \not \preceq C(tu)$---a contradiction.
\end{proof}

\subsection{Proof of Theorem~\ref{theo:main}}

We begin by generalizing our proof of EST to any number of colors.
In fact, the two proofs are nearly identical, and so we will be concise here.

\begin{lemma}\label{lemma:N3upper}
$N_3(q,n) \leq P_{q-1}(n)+1$.
\end{lemma}
\begin{proof}
Fix a $q$-coloring of the edges of $K^3_N$ that has no monochromatic monotone path of length $n$. 
We need to show that $N \leq P_{q-1}(n)$.
For every pair of vertices $u < v$ denote
$C(uv):=(1+n_1,\ldots,1+n_q)$ where $n_i$ is the length of the longest color-$i$ monotone path ending with $\{u,v\}$.
Notice that $C(u,v)\in[n]^q$.
Define $D(v)=\{x \in [n]^q : x \preceq C(uv) \mbox{ for some }u<v \}$. Since $D(v)$ is a down-set in $[n]^q$, it follows from Observation~\ref{obs:down-sets} that it suffices to show that $D(u) \neq D(v)$ for every pair of vertices.
So suppose to the contrary that $u < v$ and $D(u) = D(v)$.
By definition, $C(uv)\in D(v)$, and thus $C(uv)\in D(u)$. Hence, (again, by definition) there is a vertex $t<u$ such that $C(uv)\preceq C(tu)$.
However, the longest monochromatic monotone path ending at $\{t,u\}$ that has the same color as the edge $\{t,u,v\}$ can be extended to a longer one ending at $\{u,v\}$, implying $C(uv) \not \preceq C(tu)$---a contradiction.
\end{proof}

We now turn to prove the lower bound of Theorem~\ref{theo:main}, namely, that $N_3(q,n) > P_{q-1}(n)$. We first focus on the case of $q=2$ colors, which would simplify the notation we need.
Recall that in our proof of EST we assigned to each vertex a down-set, or equivalently, a line partition.
It therefore seems natural to do a similar thing here, and so in order to define a $2$-coloring of the edges of the complete $3$-uniform hypergraph, we identify each vertex with a distinct line partition.
Of course, we now need to define a total order on the vertex set, so we order the line partitions (i.e., vertices) lexicographically.
To be more precise: for two line partitions $A\neq B$, $A=(a_1,\ldots,a_n),B=(b_1,\ldots,b_n)$, denote $\delta(A,B)$ the smallest $i$ for which $a_i\neq b_i$; then $A$ is lexicographically smaller than $B$, denoted $A\lessdot B$, if $a_{\delta(A,B)}<b_{\delta(A,B)}$.\footnote{So for example, $(5,4,3,2,1) \lessdot (5,5,3,0,0)$.}

Our proof will follow by defining a certain coloring and showing that, roughly, if we look at the $\delta$-value of consecutive line partitions in a monochromatic monotone path, it is either strictly increasing, or else the ${\delta\text{th}}$ element of the line partitions along the path is strictly increasing.
Since this clearly cannot go on for long, any monochromatic monotone path would have to be short.

\begin{lemma}\label{lemma:N32lower}
$N_3(2,n) > P_1(n)$.
\end{lemma}
\begin{proof}
Put $N=P_1(n)$, and identify each vertex of $K^3_N$ with a distinct line partition of length $n$ with entries from $\{0,1,\ldots,n\}$, where the different line partition are ordered lexicographically.
We need to color the edges so that there is no monochromatic monotone path of  length $n$.
Given an edge whose vertices are (identified with) the three line partitions $A\lessdot B\lessdot C$, we color it black if $\delta(B,C)>\delta(A,B)$, and otherwise white.

We claim that the following holds for any monochromatic monotone path on $\ell+2$ vertices (i.e., of length $\ell$):
denoting $B\lessdot C$ 
its last two vertices, if all edges of the path are black then $\delta(B,C)> \ell$, and if all edges of the path are white then $C_{\delta(B,C)}> \ell$. This would imply that there is no monochromatic monotone path of length $n$, as required.

We prove our claim by induction on $\ell$,
noting that for the base case $\ell=0$ both conditions trivially hold.
Consider a path of length $\ell\geq 1$, whose last three vertices are $A\lessdot B\lessdot C$, and denote $\delta=\delta(B,C)$, $\delta'=\delta(A,B)$.
By the definition of our coloring, if the path is black then $\delta>\delta'\ge \ell$. Otherwise, $\delta\leq\delta'$ and so $C_{\delta}>B_{\delta}\ge B_{\delta'}\ge \ell$,
which holds since $B\lessdot C$, and since $B$ is decreasing.
\end{proof}

To generalize the above lower bound to any number of colors, we first need to define a total ordering on partitions of any given dimension.
For two $d$-dimensional partitions $A\neq B$, denote $\delta(A,B)$ the lexicographically smallest $(i_1,\ldots,i_d)\in[n]^d$ such that $A_{i_1,\ldots,i_d}\neq B_{i_1,\ldots,i_d}$.
We consider $A$ to be smaller than $B$, denoted $A\lessdot B$, if $A_{\delta(A,B)}<B_{\delta(A,B)}$.\footnote{Notice this is the lexicographic ordering if we were to "flatten" the $d$-dimensional partitions into line partitions. Thus, $\lessdot$ is clearly a total order (it is transitive, and for every $A\neq B$ either $A\lessdot B$ or else $B\lessdot A$).}

\begin{lemma}\label{lemma:N3lower}
$N_3(q,n) > P_{q-1}(n)$.
\end{lemma}
\begin{proof}
Put $d=q-1$, $N=P_d(n)$, and identify each vertex of $K^3_N$ with a distinct $n\times\cdots\times n$ $d$-dimensional partition with entries from $\{0,1,\ldots,n\}$.
Further, order the vertex set using the above defined $\lessdot$.
We need to color the edges with the colors $\{1,2,\ldots,d+1\}$ so that there is no monochromatic monotone path of length $n$. 
For every three $d$-dimensional partition $A\lessdot B\lessdot C$,
if there is an $1 \le i \le d$ satisfying $(\delta(B,C))_i>(\delta(A,B))_i$ we color $\{A,B,C\}$ by $i$ (if there are several such $i$ we choose one arbitrarily); if there is no such $i$, we color the edge by $d+1$.

We claim that the following holds for any monochromatic monotone path on $\ell+2$ vertices (i.e., of length $\ell$):
denoting $B\lessdot C$ 
its last two vertices and denoting $\delta(B,C)=(\delta_1,\ldots,\delta_d)$,
if all edges of the path are colored by $1\le i\le d$ then $\delta_i > \ell$, and if all edges of the path are colored by $d+1$ then $C_{\delta_1,\ldots,\delta_d} > \ell$.
This would imply that there is no monochromatic monotone path of length $n$.
%
We prove our claim by induction on $\ell$, noting that for the base case $\ell=0$ both conditions trivially hold. Consider a path of length $\ell\geq 1$ whose last three vertices are $A\lessdot B\lessdot C$, and denote
$(\delta_1,\ldots,\delta_d)=\delta(B,C)$, $(\delta'_1,\ldots,\delta'_d)=\delta(A,B)$.
By the definition of our coloring, if the path is colored by $1\le i\le d$ then
$\delta_i>\delta'_i\ge \ell$; otherwise,
$(\delta_1,\ldots,\delta_d)\preceq(\delta'_1,\ldots,\delta'_d)$ and so
$C_{\delta_1,\ldots,\delta_d}>B_{\delta_1,\ldots,\delta_d}\ge B_{\delta'_1,\ldots,\delta'_d}\ge \ell$,
which holds since $B\lessdot C$, and since $B$ is decreasing in each line.
This completes the proof.
\end{proof}

\begin{proof}[Proof of Theorem~\ref{theo:main}]
It follows immediately from Lemma~\ref{lemma:N3upper} and Lemma~\ref{lemma:N3lower} that $N_3(q,n)=P_{q-1}(n)+1$.
\end{proof}

Let us briefly mention the implications of the above arguments to a natural extension of $N_3(q,n)$.
Let $N_3(q,n_1,\ldots,n_q)$ be the smallest integer $N$ so that every coloring of the edges of $K^3_N$ using $q$ colors contains a color-$i$ monotone path of length $n_i$, in at least one of the colors $1\le i\le q$.
One can modify both the upper and lower bound proofs above in a straightforward manner to show that $N_3(q,n_1,\ldots,n_q)=P_{q-1}(n_1,\ldots,n_q)+1$, where $P_d(n_1,\ldots,n_d,n)$ is the number of $d$-dimensional $n_1\times\cdots\times n_d$ (hyper)matrices with entries from $\{0,1,\ldots,n\}$ that decrease in each line.
One can easily generalize the argument proving (\ref{eq:dim1partition}) to show $P_1(a,b)=\binom{a+b}{a}$, which implies the known result  $N_3(2,a,b)=\binom{a+b}{a}+1$.
As for the next case, MacMahon~\cite{MacMahon16} proved a result which is in fact more general than the one stated in~(\ref{eq:dim2partition}), namely,
$
P_2(a,b,c)=\prod_{i=1}^{a}\prod_{j=1}^{b}\prod_{k=1}^{c} \frac{i+j+k-1}{i+j+k-2} \;,
$
which gives an exact bound for $N_3(3,a,b,c)$.

\subsection{Proof of Theorem~\ref{theo:transform}}

It will be more convenient to prove the following stronger bound for every $k \geq 4$:
\begin{equation}\label{strongerupper}
N_k(q,n) \leq N_{k-2}(N_3(q,n)-1,2) \;.
\end{equation}
\begin{proof}
Fix a $q$-coloring of the edges of $K^k_N$ that has no monochromatic monotone path of length $n$. 
We need to show that $N < N_{k-2}(N_3(q,n)-1,2)$.
For any $k-1$ vertices $x_1<\cdots<x_{k-1}$ denote
$C(x_1,\ldots,x_{k-1}):=(1+n_1,\ldots,1+n_q)$ where $n_i$ is the length of the longest color-$i$ monotone path ending with $\{x_1,\ldots,x_{k-1}\}$, and notice that $C(x_1,\ldots,x_{k-1})\in[n]^q$.
For any $k-2$ vertices $x_2<\cdots<x_{k-1}$ we define
$$D(x_2,\ldots,x_{k-1})=\{y \in [n]^q : y \preceq C(x_1,x_2,\ldots,x_{k-1}) \mbox{ for some }x_1<x_2 \} \;,$$
which is of course a down-set in $[n]^q$.

Now, we define a coloring of the complete $(k-2)$-uniform hypergraph $K^{k-2}_N$ (on the same vertex set) by letting the color of an edge $\{x_1,\ldots,x_{k-2}\}$ be $D(x_1,\ldots,x_{k-2})$.
We claim that there is no monochromatic monotone path of length $2$ in our coloring of $K^{k-2}_N$.
Indeed, suppose for contradiction that $\{x_1,\ldots,x_{k-2}\}$ and $\{x_2,\ldots,x_{k-1}\}$ receive the same color, that is, $D(x_1,\ldots,x_{k-2})=D(x_2,\ldots,x_{k-1})$.
By definition, $C(x_1,\ldots,x_{k-1})\in D(x_2,\ldots,x_{k-1})$, and so by our assumption, $C(x_1,\ldots,x_{k-1})\in D(x_1,\ldots,x_{k-2})$.
Hence, (again, by definition) there is a vertex $x<x_1$ such that $C(x_1,\ldots,x_{k-1})\preceq C(x,x_1,\ldots,x_{k-2})$.
However, in the (given) coloring of $K^k_N$, the longest monochromatic monotone path ending at $\{x,x_1,\ldots,x_{k-2}\}$ that has the same color as the edge $\{x,x_1,\ldots,x_{k-1}\}$ can be extended to a longer one ending at $\{x_1,\ldots,x_{k-1}\}$, so $C(x_1,\ldots,x_{k-1}) \not \preceq C(x,x_1,\ldots,x_{k-2})$---a contradiction.

We conclude that in our coloring of $K^{k-2}_N$ there is no monochromatic monotone path of two edges. 
Therefore it must be the case that $N < N_{k-2}(c,2)$ where $c$ is the number of colors we used.
Since each $D(x_1,\ldots,x_{k-2})$ is a down-set in $[n]^q$,
Observation~\ref{obs:down-sets} implies $c\leq P_{q-1}(n)$, and since Theorem~\ref{theo:main} tells us that $P_{q-1}(n)=N_3(q,n)-1$ the proof is complete.
\end{proof}

Having completed the proof of (\ref{strongerupper}), let us prove that it implies the (weaker) bound 
$N_k(q,n) \leq t_{k-2}(N_3(q,n))$ stated in Theorem~\ref{theo:transform}.
We proceed by induction on $k$, noting that when $k=3$ there is nothing to prove, and that the case $k=4$ follows from (\ref{strongerupper}) and the fact that $N_2(q,2) \leq 2^q+1$ (see the discussion preceding (\ref{eq:PathEST})).
For the induction step, assuming $k \geq 5$, we have from (\ref{strongerupper}) that
$N_k(q,n) \leq N_{k-2}(q',2)$, where $q'=N_3(q,n)$, and by the induction hypothesis, $N_{k-2}(q',2) \leq t_{k-4}(N_3(q',2))$.
Applying the upper bound in Corollary~\ref{coro:bound} we deduce that
\begin{eqnarray*}
N_k(q,n) &\leq& t_{k-4}(N_3(q',2))\\
&\leq& t_{k-3}(2\cdot 2^{q'-1})\\
&=& t_{k-2}(q')\\
&=&  t_{k-2}(N_3(q,n))\;,
\end{eqnarray*}
as needed.


\subsection{An improved lower bound for $P_d(n)$}\label{subsec:PdnBound}

Recall that by Observation~\ref{obs:down-sets}, $P_{d-1}(n)$ is the number of down-sets in $[n]^d$.
Here it will be convenient to use the notion dual to that of a down-set.
A set $A\subseteq[n]^d$ is an \emph{antichain} if $a\in A$ implies $x\notin A$ for all $x\preceq a$.
In this section (and later on as well) it will
be useful to refer to the following simple observation.

\begin{obs}\label{obs:antichains}
The number of antichains $A\subseteq[n]^d$ is $P_{d-1}(n)$.
\end{obs}

\begin{proof}
Observe that retaining only the $\prec$-maximal elements of a down-set yields a unique antichain. Moreover, every antichain can clearly be obtained from a down-set in this manner. Hence down-sets are in bijection with antichains, and from Observation~\ref{obs:down-sets} we deduce that the number of antichains in $[n]^d$ is exactly $P_{d-1}(n)$.

\end{proof}

We will use the antichain characterization of Observation \ref{obs:antichains} in order to prove a lower bound on $P_{d}(n)$.
First, we need a simple lemma.
For any $d\leq k\leq dn$ let $S_n(k,d)$ be the number of solutions to the equation $x_1+\cdots+x_d=k$ with $x_i\in[n]$ 
for every $1 \leq i \leq d$.
Notice that $S_2(k,d)$ are simply the binomial coefficients (specifically, $S_2(k,d)=\binom{d}{k-d}$).
The numbers $S_n(k,d)$ (which were already studied by Euler) satisfy many of the properties of binomial coefficients, e.g., symmetry about the middle $S_n(k,d)=S_n(dn-k,d)$.
In particular, it is known that $\max_k S_n(k,d)=:M_{d,n}$ is achieved at the middle, that is, at $k=d(n+1)/2$ (this already follows from the work of de Bruijn et al.~\cite{deBruijnTeKr51}).
We remark that it can be shown (using a version of the central limit theorem, see~\cite{MattnerRo08}) that when $d$ tends to infinity,
$$ 
M_{d,n} = \sqrt\frac{6}{\pi(n^2-1)d}\cdot n^d(1+o_d(1)) \quad
 \Big(\approx \sqrt\frac{6}{\pi}\cdot\frac{n^{d-1}}{\sqrt{d}}\Big) \;.
$$ 
However, we will need to estimate $M_{d,n}$ for \emph{any} $d$, and so we prove the following easy lower bound.

\begin{lemma}\label{lemma:midrankLower}
For every $d \geq 1$ and $n\ge 1$ we have
$$
M_{d,n} = \max_k S_n(k,d) \ge \frac23\cdot\frac{n^{d-1}}{\sqrt{d}} \;.
$$
\end{lemma}

\begin{proof}
Let $x_1,\ldots,x_d$ be randomly (independently and uniformly) chosen from $[n]$, and set $X=x_1+\cdots+x_d$. Notice that the expectation of $X$ is $\mu:=d(n+1)/2$, and its variance is $\sigma^2:=d(n^2-1)/12 \le dn^2/12$.
By Chebyschev's Inequality, for $\lambda>0$ to be chosen later,
${\Pr\big(\lvert X-\mu\rvert \ge \lambda\sigma\big)} \le \lambda^{-2}.$
Put $\alpha=1-\lambda^{-2}$. Since $\Pr(X=k)=S_n(k,d)/n^d$, we have
$$
\sum_{k=\lceil\mu-\lambda\sigma\rceil}^{\lfloor\mu+\lambda\sigma\rfloor} S_n(k,d) \ge \alpha n^d \;.
$$
Using the pigeonhole principle, we can bound from below the largest $S_n(k,d)$ in the range above, which is $M_{d,n}$, by
$$
\frac{\alpha n^d}{2\lambda\sigma} \geq \frac{\sqrt{3}\alpha}{\lambda}\cdot\frac{n^{d-1}}{\sqrt{d}} \;.
$$
Choosing $\lambda=\sqrt3$ so as to maximize $\alpha/\lambda$ yields the desired bound.
\end{proof}

We now use Lemma~\ref{lemma:midrankLower} above to deduce a lower bound
on $P_{d-1}(n)$.

\begin{proof}[Proof of Theorem~\ref{theo:lower}]
We show that
\begin{equation}\label{eq:Pdlower}
P_{d-1}(n) \geq 2^{M_{d,n}} \;, 
\end{equation}
which, by Lemma~\ref{lemma:midrankLower}, would complete the proof.
For $x=(x_1,\ldots,x_d)\in[n]^d$ write $\size{x}=\sum_{i=1}^d x_i$, and
note that $S_n(k,d)$ is the number of $x$ satisfying $\size{x}=k$.
Let $d\leq k\leq dn$, and let $A\subseteq[n]^d$ be a set whose every member $a$ satisfies $\size{a}=k$. Then $A$ is an antichain since every $x\preceq a$, where $a\in A$ and $x\neq a$, must satisfy $\size{x}<\size{a}=k$, implying that $x\notin A$.
It follows that the number of antichains in $[n]^d$ is at least $2^{S_n(k,d)}$. Taking $k$ so as to maximize $S_n(k,d)$ shows that the number of antichains in $[n]^d$ is at least $2^{M_{d,n}}$. Observation~\ref{obs:antichains} now proves (\ref{eq:Pdlower}), as needed.
\end{proof}

\section{General Hypergraphs}\label{sec:kgraphs}

Our goal in this section is to give an enumerative characterization of $N_k(q,n)$ analogous to the one we have previously obtained for $N_3(q,n)$ in Theorem~\ref{theo:main}.
We will show that just as the numbers $N_3(q,n)$ are closely related to high-dimensional integer partitions, the numbers $N_k(q,n)$ are closely related to what can naturally be thought of as higher-order analogues of partitions.
For simplicity of presentation, we will focus on higher-order \emph{line} partitions, which we would use to characterize $N_k(2,n)$. The characterization for $N_k(q,n)$ is obtained by exactly the same arguments.

\begin{figure}
\centering
\includegraphics[scale=0.5]{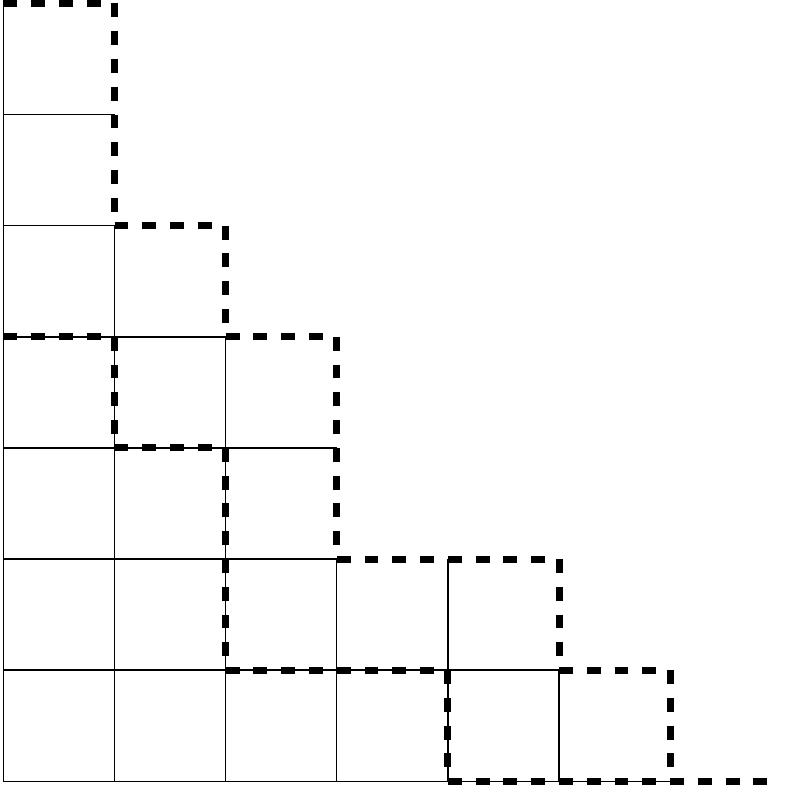}
\caption{A line partition contained in another.}
\label{fig:1partitionInclusion}
\end{figure}

To state our characterization we will need to restate Observation~\ref{obs:down-sets} in a language that will be somewhat easier to generalize.
So for what follows, let us set $\P^2[n] = [n]^2$, and let $\P^3[n]$ denote the family of line partitions as defined in Section~\ref{sec:intro} (so $\size{\P^3[n]}=P_1(n)$).
Recall that by Observation~\ref{obs:down-sets}, line partitions are in bijection with down-sets in $[n]^2$; put in other words, every line partition $\F \in \P^3[n]$ is a subset of $\P^2[n]$ with the property that if $x \in \F$ then so is every $x' \preceq x$. Now, to make the above easier to generalize, we think of any $x=(x_1,x_2) \in[n]^2$ as a multiset (with $x_i-1$ being the multiplicity of the $i^{th}$ element). Accordingly, we henceforth use the notation $x' \subseteq x$ instead of $x' \preceq x$.\footnote{For consistency, one can instead think of $x$ as a (standard) set, say by using the ``unary representation'': Letting $A,B$ be two disjoint ordered sets of cardinality $n$, we may represent $x$ as the subset of $A\cup B$ consisting of the first $x_1$ members of $A$ and the first $x_2$ members of $B$.}

We define $\P^4[n]$ in a similar manner, so that the members of $\P^4[n]$ are down-sets in $\P^3[n]$;
that is, each $\F \in \P^4[n]$ is a collection of line partitions (i.e., elements of $\P^3[n]$) with the property that if a line partition $L$ belongs to $\F$ then so do all line partitions $L' \subseteq L$. Note that the way we have redefined line partitions in the previous paragraph allows
us to talk about one line partition being a subset of another. For the sake of clarity the reader can check Figure~\ref{fig:1partitionInclusion} which depicts inclusion of two line partitions; furthermore, see Section~\ref{sec:conclude} for a detailed discussion about $\P^4[n]$.

In general, we inductively define $\P^k[n]$, whose members are the down-sets in $\P^{k-1}[n]$, as follows.
\begin{definition}\label{def:highOrder}
Let $\P^2[n] = [n]^2$ and suppose we have already defined $\P^{k-1}[n]$. A set $\F\subseteq\P^{k-1}[n]$ is in $\P^{k}[n]$ if $S\in\F$ implies $S'\in\F$ for any $S' \subseteq S$.
\end{definition}

We denote the cardinality of $\P^k[n]$ by $\rho_k(n)$.
In what follows, we will sometimes refer to the members of $\P^k[n]$ as {\em line partitions of order $k$}.
We obviously have $\rho_k(n)\leq 2^{\rho_{k-1}(n)}$ for any $k \geq 3$, and in particular, $\rho_4(n)\le 2^{\binom{2n}{n}}\le 2^{2^{2n}}$. Summarizing, we have
$$
\rho_2(n)=n^2, \qquad \rho_3(n)=\binom{2n}{n}, \qquad \rho_k(n) \leq t_{k-1}(2n)
$$
for any $k\geq 3$.
Our characterization of $N_k(2,n)$ is thus the following.

\begin{theo}\label{theo:Nk2exact} For every $k \geq 2$ and $n \geq 2$ we have
$$
N_k(2,n) = \rho_k(n)+1\;.
$$
\end{theo}

Note that Theorem~\ref{theo:Nk2exact} subsumes both ESL and EST (as well as our main result in Theorem~\ref{theo:main} when $q=2$; see (\ref{eq:general}) at the end of this section for arbitrary $q$).
The proof of Theorem~\ref{theo:Nk2exact} follows immediately from Lemmas~\ref{lemma:Nkupper} and~\ref{lemma:Nklower} stated below.

\begin{lemma}\label{lemma:Nkupper} 
$N_k(2,n)\le \rho_k(n)+1.$
\end{lemma}

\begin{proof}
Fix a $2$-coloring of the edges of $K^k_N$ that has no monochromatic monotone path of length $n$. 
We need to show that $N \leq \rho_k(n)$.
We begin with some definitions.
For every $k-1$ vertices $x_1<\cdots<x_{k-1}$ denote $D(x_1,\ldots,x_{k-1}):=(1+n_b,1+n_w)$ where $n_b$ is the length of the longest black monotone path ending with $\{x_1,\ldots,x_{k-1}\}$, and $n_w$ is defined in a similar way only with respect to white monotone paths.
Notice that $D(x_1,\ldots,x_{k-1})\in\P^2[n]$.
For any $r \in \{k-2,k-1,\ldots,1\}$ and any $r$ vertices $x_{k-r}<\cdots<x_{k-1}$
we recursively define
$$
D(x_{k-r},\ldots,x_{k-1})=\{S\in\P^{k-r}[n] : S\subseteq D(x,x_{k-r},\ldots,x_{k-1}) \text{ for some } x<x_{k-r}\} \;.
$$
Since for $r=k-1$ we have $D(x_{k-r},\ldots,x_{k-1})\in\P^2[n]$, from  Definition~\ref{def:highOrder} we immediately have that $D(x_{k-r},\ldots,x_{k-1})\in\P^{k-r+1}[n]$.
In particular, it holds that $D(v)\in\P^{k}[n]$.
Hence to complete the proof it suffices to show that $D(u) \neq D(v)$ for every pair of vertices.

We first prove that for any $2\le r\le k-1$ and any $r$ vertices $x_{k-r}<\cdots<x_{k-1}$,
\begin{equation}\label{eq:recursiveD}
D(\x,x_{k-1})\subseteq D(x_{k-r},\x) \quad\Rightarrow\quad
\exists x<x_{k-r} \mbox{ s.t. }
D(x_{k-r},\x,x_{k-1})\subseteq D(x,x_{k-r},\x) \;,
\end{equation}
where $\x$ is short for $x_{k-r+1},\ldots,x_{k-2}$.
By definition, $D(x_{k-r},\x,x_{k-1})\in D(\x,x_{k-1})$.
Therefore, the assumption that $D(\x,x_{k-1})\subseteq D(x_{k-r},\x)$ implies $D(x_{k-r},\x,x_{k-1})\in D(x_{k-r},\x)$. Hence, (again, by definition) there is some vertex $x<x_{k-r}$ such that $D(x_{k-r},\x,x_{k-1})\subseteq D(x,x_{k-r},\x)$, as desired. 

By iteratively applying (\ref{eq:recursiveD}) we deduce the following.
Suppose there are two vertices $u<v$ such that $D(u)=D(v)$. Note that the case $r=2$ of (\ref{eq:recursiveD}) is satisfied, namely, $D(v)\subseteq D(u)$, where we put $x_{k-1}=v,x_{k-2}=u$. Therefore, the conclusion in (\ref{eq:recursiveD}) for the case $r=k-1$ must hold as well. That is, there are $k$ vertices $x_{0}<\cdots<x_{k-3}<u<v$ such that
$D(x_{1},\ldots,u,v) \subseteq D(x_{0},\ldots,u)$ ($\in[n]^2$).
However, the longest monochromatic monotone path ending at $\{x_{0},\ldots,u\}$ that has the same color as the edge $\{x_{0},\ldots,u,v\}$ can be extended to a longer one ending at $\{x_{1},\ldots,u,v\}$, implying $D(x_{1},\ldots,u,v) \not \subseteq D(x_{0},\ldots,u)$---a contradiction.
This completes the proof.
\end{proof}

Before giving the lower bound on $N_k(2,n)$ matching our upper bound, we first give a short proof for the case of graphs, that is, for the fact that $N_2(2,n)\geq n^2+1$, which would be suggestive of the generalization we plan to make. In other words, we give a black/white coloring of the complete graph on $n^2$ vertices that has no monochromatic monotone path of length $n$.
First, identify each of the $n^2$ vertices with a distinct pair of integers $(x,y)\in[n]^2$, 
and order the pairs (i.e., vertices) lexicographically.
For an edge whose vertices are identified with the pairs $(x_1,y_1)$ and $(x_2,y_2)$, where $(x_1,y_1)$ is lexicographically smaller than $(x_2,y_2)$,
color it black if $x_1<x_2$; otherwise, color it white if $y_1<y_2$.
Crucially, observe that the lexicographic ordering ensures that every edge is indeed colored by either black or white.
Under this coloring, it is clear that any monochromatic monotone path is of length at most $n-1$, or equivalently, has at most $n$ vertices, simply because the pairs along the path are strictly increasing either in the first or in the second coordinate.

Let us give the necessary definitions we will need in order to prove the lower bound on $N_k(2,n)$ stated in Lemma~\ref{lemma:Nklower} below.
We begin by extending the lexicographic ordering of pairs of nonnegative integers (which are line partitions of order $2$) to partitions of arbitrary order.
Assuming we have already defined lexicographic ordering for line partitions of order $k-1$,
we say that an order-$k$ line partition $\F$ is lexicographically smaller than an order-$k$ line partition $\F'$, denoted $\F\lessdot\F'$, if the lexicographically first element on which they differ (that is, the lexicographically first member of the symmetric difference $\F\triangle\F'$) is in $\F'$.\footnote{The reader may find it useful here to think of an order-$k$ line partition as a $0/1$-vector indexed by order-$(k-1)$ line partitions that are ordered lexicographically.
Now, $\F$ is lexicographically smaller than $\F'$ precisely when $0=\F_S<\F'_S=1$ where $S$ is the first index in which the two vectors differ.}
(This ordering in fact coincides with the one preceding the proof of Lemma~\ref{lemma:N32lower} that we used for the $k=3$ case.)


Furthermore, when two order-$k$ line partition $\F,\F'$ satisfy $\F\nsupseteq\F'$ we denote $\delta(\F,\F')$ the lexicographically first element that $\F'$ contains and $\F$ does not (that is, the lexicographically first member of $\F'\setminus\F$).\footnote{In the vector terminology mentioned earlier, $\delta(\F,\F')$ is the first index $S$ such that $0=\F_S<\F'_S=1$.}
Note that $\delta(\F,\F')$ is a line partition of order $k-1$; that is, if $\F,\F'\in\P^k[n]$ then $\delta(\F,\F')\in\P^{k-1}[n]$.

\begin{claim}\label{claim:threePartitions}
For any sequence of order-$k$ line partitions $\F_1\nsupseteq\F_2\nsupseteq\cdots\nsupseteq\F_t$ we have
$\delta(\F_1,\F_2)\nsupseteq\delta(\F_2,\F_3)\nsupseteq\cdots\nsupseteq\delta(\F_{t-1},\F_t)$.
\end{claim}
\begin{proof}
It suffices to show that for any three order-$k$ line partitions $\F_1\nsupseteq\F_2\nsupseteq\F_3$ we have $\delta(\F_1,\F_2)\nsupseteq\delta(\F_2,\F_3)$.
By definition, $\delta(\F_1,\F_2)\in\F_2$ and $\delta(\F_2,\F_3)\notin\F_2$. Since $\F_2$ is closed under taking subsets, it cannot be the case that  $\delta(\F_2,\F_3)\subseteq\delta(\F_1,\F_2)$.
\end{proof}

\begin{lemma}\label{lemma:Nklower} 
$N_k(2,n)>\rho_k(n)$.
\end{lemma}

\begin{proof}
Put $N=\rho_k(n)$, and identify the vertex set of $K^k_N$ with $\P^{k}[n]$, ordered lexicographically.
We need to color the edges so that there is no monochromatic monotone path of length $n$.
Let $e=\{\F_1,\ldots,\F_k\}$ be an edge, and note that since the $\F_i$'s are lexicographically ordered $\F_1\lessdot\cdots\lessdot\F_k$ it follows that $\F_1\not\supseteq\cdots\not\supseteq\F_k$.
By Claim~\ref{claim:threePartitions} we have $\delta(\F_1,\F_2)\nsupseteq\cdots\nsupseteq\delta(\F_{k-1},\F_k)$, and observe that we may apply Claim~\ref{claim:threePartitions} again, this time on the sequence of $\delta$'s, which is a sequence of $k-1$ line partitions of order $k-1$. By applying Claim~\ref{claim:threePartitions} $i$ times in a similar fashion we obtain a sequence of $k-i$ line partitions in $\P^{k-i}[n]$. In particular, after $i=k-2$ applications we obtain two pairs $(x_1,y_1)=:\delta^*(\F_1,\ldots,\F_{k-1})$ and $(x_2,y_2)=:\delta^*(\F_2,\ldots,\F_{k})$ that belong to $\P^{2}[n]$~($=[n]^2$) and satisfy $(x_1,y_1)\not\supseteq(x_2,y_2)$.
We color the edge $e$ black if $x_1<x_2$; otherwise, we necessarily have $y_1<y_2$, and we color $e$ white.

Observe that a monochromatic monotone path of length $\ell$ determines a sequence of $\ell+1$ pairs from $[n]^2$
(namely, if the path is on the vertices $\F_1\lessdot\cdots\lessdot\F_{\ell+k-1}$ then it determines the $\ell+1$ pairs $\delta^*(\F_1,\ldots,\F_{k-1}),\delta^*(\F_2,\ldots,\F_k),\ldots,\delta^*(\F_{\ell+1},\ldots,\F_{\ell+k-1})$). Moreover, these pairs strictly increase either in the first or in the second coordinate. We deduce that $\ell+1 \leq n$, completing the proof.
\end{proof}

Having completed the proof of Theorem~\ref{theo:Nk2exact} we finally mention that all the above can be extended in a straightforward manner to any number $q>2$ of colors, so as to determine $N_k(q,n)$.
Setting $\P^2_d[n]=[n]^d$, we define $\P^k_d[n]$ so that its members are the down-sets in $\P^{k-1}_d[n]$, similarly to what we did before.
Denoting $\rho_{k,d}(n)=\size{\P^k_d[n]}$, we of course have
$\rho_{3,d}(n) = P_{d-1}(n)$,
and since for any $k\geq 3$,
$\rho_{k,d}(n)\leq 2^{\rho_{k-1,d}(n)}$,
we have $\rho_{k,d}(n) \leq t_{k-2}(P_{d-1}(n))$.
By following the proofs in this section essentially line by line, and replacing
$[n]^2$ with $[n]^q$, one obtains the characterization
\begin{equation}\label{eq:general}
N_k(q,n)= \rho_{k,q}(n) + 1 \;.
\end{equation}


We now use the above characterization to deduce a recursive lower bound on $N_{k}(q,n)$.

\begin{lemma}\label{lemma:NkReclower}
For every $k \geq 4$, $q\geq 2$, and $n \geq 2$ we have
$$N_k(q,n) \geq 2^{N_{k-1}(q,n)/N_{k-2}(q,n)} \;.$$
\end{lemma}
\begin{proof}
We show that for every $k\geq 4$, $d\geq 2$, and $n\geq 2$, we have
$\rho_{k,d}(n)\geq 2^{(\rho_{k-1,d}(n)+1)/(\rho_{k-2,d}(n)+1)}$, from which the proof immediately follows using (\ref{eq:general}).
Put $U=\rho_{k-2,d}(n)$.
Let $L_i$ be the number of sets $A\in\P^{k-1}_d[n]$ of cardinality $i$, for some $0\le i\le U$.
Observe that any collection of such sets $A$, together with all $A'\subseteq A$, determines a distinct down-set; indeed, any such down-set has a unique set of maximal elements. It therefore follows that $\rho_{k,d}(n) \geq 2^{L_i}$.
Now, take $i$ so as to maximize $L_i$.
Since $\rho_{k-1,d}(n) = \sum_{j=0}^{U} L_j$, we have $L_i\geq \rho_{k-1,d}(n)/(U+1)$, and in fact, this inequality is clearly strict. The desired inequality now follows.
\end{proof}


We will show that Theorem~\ref{theo:transformlower}
follows by iteratively applying Lemma~\ref{lemma:NkReclower}.
First, we need a simple lemma concerning differences of towers.
We henceforth write $\log()$ for $\log_2()$.

\begin{lemma}\label{lemma:towerDiff}
For any $k\geq 2$ and positive reals $a\geq b+1$, $a\geq 3$, we have  $t_k(a)-t_k(b) \geq t_k(a - 2^{-(k-2)})$.
\end{lemma}
\begin{proof}
We proceed by induction on $k$. For the base case $k=2$ we have $2^a-2^b \geq 2^a-2^{a-1}=2^{a-1}$, as needed.
Now, observe that it follows from the standard estimate $1-p\geq e^{-2p}$, applicable for any $0\leq p\leq \frac12$, that for every pair of positive reals $x\geq 2y$ we have $x-y=x(1-y/x)\geq xe^{-2y/x}$.
So for every $x\geq 2y>0$ we have
\begin{equation}\label{eq:logs}
\log (x-y) \geq \log x - 2\log(e)y/x \geq \log x - 3y/x \;.
\end{equation}
Denote by $\log^{(i)}()$ the $i\geq 1$ times iterated $\log()$ (so, e.g., $\log^{(2)}(x)=\log\log(x)$).
For the induction step, which is equivalent to $\log^{(k-1)}(t_k(a)-t_k(b))\geq a - 2^{-(k-2)}$, we have
\begin{eqnarray*}
\log^{(k-1)}(t_k(a)-t_k(b)) &=& \log\log^{(k-2)}(t_{k-1}(2^a)-t_{k-1}(2^{b}))\\
&\geq& \log(2^a - 2^{-(k-3)})\\
&\geq& a - 3\cdot 2^{-(k-3)}/2^a\\
&\geq& a - \frac12\cdot 2^{-(k-3)} \;,
\end{eqnarray*}
where the first inequality follows from the induction hypothesis, the second inequality from (\ref{eq:logs}), and the third inequality from the assumption that $a\geq 3$.
The proof follows.
\end{proof}

\begin{proof}[Proof of Theorem~\ref{theo:transformlower}]
The case $k=4$ follows immediately from Theorem~\ref{theo:transformlower} and the fact that $N_2(q,n)\leq n^q+1$ (see the discussion preceding (\ref{eq:PathEST})).
We prove by induction on $k\geq 5$ that $N_k(q,n)\geq t_{k-2}(\frac{N_3(q,n)}{2n^q} - \sum_{i=0}^{k-6}2^{-i})$, from which the result clearly follows.
We assume throughout the proof that $n$ is larger than some absolute constant (i.e., independent of $q,k$).
For the base case $k=5$ we get
\begin{eqnarray*}
\log N_5(q,n) &\geq& N_4(q,n)/N_3(q,n)\\
&\geq& 2^{N_3(q,n)/(n^q+1)}/N_3(q,n)\\
&\geq& 2^{N_3(q,n)/2n^q} \;,
\end{eqnarray*}
where the first and second inequalities rely on Lemma \ref{lemma:NkReclower}, and the last inequality is due to the fact that 
\begin{equation}\label{eq:N3overN2}
\log_2 N_3(q,n) = o_n(N_3(q,n)/n^q)
\end{equation}
uniformly for all $n$ larger than some absolute constant (i.e., independently of $q$; this easily follows from the bounds in Corollary~\ref{coro:bound}).
%
For the induction step, assuming $k\geq 6$, we have
\begin{eqnarray*}
\log N_k(q,n) &\geq& N_{k-1}(q,n)/N_{k-2}(q,n)\\
&\geq& N_{k-1}(q,n)/t_{k-4}(N_3(q,n))\\
&\geq& 2^{t_{k-4}(\frac{N_3(q,n)}{2n^q}-\sum_{i=0}^{k-7}2^{-i}) - t_{k-4}(\log N_3(q,n))}\\
&\geq& 2^{t_{k-4}(\frac{N_3(q,n)}{2n^q}-\sum_{i=0}^{k-7}2^{-i} - 2^{-(k-6)})}\\
&=& t_{k-3}\Big(\frac{N_3(q,n)}{2n^q}-\sum_{i=0}^{k-6}2^{-i}\Big)\;,
\end{eqnarray*}
where the first inequality follows from Lemma~\ref{lemma:NkReclower}, the second inequality from the upper bound $N_k(q,n) \leq t_{k-2}(N_3(q,n))$ in Theorem~\ref{theo:transform}, the third inequality from the induction hypothesis,
and the fourth inequality from Lemma~\ref{lemma:towerDiff} using the fact that $k-4\geq 2$, as well as (\ref{eq:N3overN2}). This completes the proof.
\end{proof}

\section{Concluding Remarks and Open Problems}\label{sec:conclude}


\paragraph{An exact bound for the two-edge path?:} Combining Theorem~\ref{theo:main} together with Observation~\ref{obs:antichains} we see that $N_3(q,n)$ is determined by the number of antichains in $[n]^{q}$.
In particular, $N_3(q,2)-1$ is the number of antichains in $[2]^{q}$, where $[2]^{q}$ is the Boolean poset over $[q]$, that is, the poset of the subsets---ordered by inclusion---of a $q$-element ground set.
These numbers are called {\em Dedekind numbers}, first introduced by Dedekind in~\cite{Dedekind87}, and despite much research they have no known (reasonable) closed formula.
Hence, we do not expect a closed formula even for $N_3(q,2)$ (i.e., the case of two-edge paths), which is the simplest non-trivial case (for any given $q$).

\paragraph{The exact exponent of $N_3(q,n)$:} While Dedekind numbers do not have a closed formula, Kleitman~\cite{Kleitman69} has shown that the number of antichains in $[2]^d$ is at most $2^{(1+o(1))\binom{d}{d/2}}$. Note that $2^{\binom{d}{d/2}}$ is a trivial lower bound since any family of subsets of $[d]$, all of the same cardinality $d/2$, is an antichain.
So Kleitman's result can be phrased as saying that the size of the largest antichain in $[2]^d$ essentially determines also the number of antichains.
It is known that the poset $[n]^d$ (equipped with the partial order $\preceq$) satisfies the so-called ``Sperner property'', which means here that the largest antichain in $[n]^d$ is
also the ``middle layer''. As noted in Subsection~\ref{subsec:PdnBound}, the middle layer is of size $(1+o_d(1))\sqrt{6/d\pi}\cdot n^{d-1}$ for
$d,n \gg 1$, 
so $P_{d-1}(n)$ is (at least) exponentially larger than this.\footnote{So we also note that the constant $2/3$ in Theorem~\ref{theo:lower} can be improved to $\sqrt{6/\pi}$ when $d$ is large.} Now, recall that our best upper bound on $P_{d-1}(n)$ is (roughly) $2^{2n^{d-1}}$, so the two exponents are off by a $\sqrt{d}$ factor.

As the problem of estimating the number of antichains in $[n]^d$ seems like a natural extension of the classical problem of estimating the Dedekind numbers (which is the number of antichains in $[2]^d$),
the following question seems especially natural---does the phenomenon in Dedekind's problem hold in this case as well? In other words, is the number of antichains in $[n]^{d}$ of order $2^{cn^{d-1}/\sqrt{d}}$? One might even be bolder and ask if the exact constant in the exponent is $\sqrt{6/\pi}$ (as the one in the size of the largest antichain). Such a bound would immediately imply that  $N_3(q,n)=2^{(1+o(1))n^{q-1}/\sqrt{\pi q/6}}$ with the $o(1)$ term going to $0$ as $q \rightarrow \infty$.

We note that by now there are several proofs~\cite{Kahn02,KleitmanMa75,Korshunov81,Pippenger99} of Kleitmen's result\footnote{We remark that earlier works, such as the one by Hansel~\cite{Hansel66}, proved the weaker result that the number of antichains in $[2]^{n}$ is at most $C^{\binom{n}{n/2}}$ for some absolute constant $C$. Observe, however, that proving a similar result in our setting, that is, that the number of antichains in $[n]$ is at most $C^{n^{d-1}/\sqrt{d}}$ will suffice to prove that $P_d(n) \leq 2^{O(n^{d-1}/\sqrt{d})}$.} but as of now we are not able to apply any of them
in order to prove that the number of antichains in $[n]^d$ is of order $2^{O(n^{d-1}/\sqrt{d})}$.
It is worth mentioning here that an antichain in the poset $[n]^d$ is nothing but an independent set in the graph whose vertex set is $[n]^d$ and with an edge between any two $x,y\in[n]^d$ satisfying $x\preceq y$ or $y\preceq x$ (this is the corresponding comparability graph).
So it seems that the next step towards a complete solution of the Ramsey-type problem considered in this paper is yet another classical enumerative problem, namely, that of counting independent sets.

\paragraph{Tighter bounds for $N_4(2,n)$:} While we know that $N_3(2,n)=\binom{2n}{n}+1$, it seems hard to determine even the asymptotics of $N_k(2,n)$ for $k > 3$. Let us explain why this is the case by focusing on $k=4$.
Recall that by Corollary~\ref{coro:matousek} we know that $N_4(2,n)=2^{2^{(2-o(1))n}}$. However, as for the \emph{logarithm} of $N_4(2,n)$, what we know by Theorem~\ref{theo:transform} and Lemma~\ref{lemma:NkReclower} is only that (roughly)
$$\dbinom{2n}{n}/n^2 \leq \log_2 N_4(2,n) \leq \dbinom{2n}{n} \;.$$
%
Actually, the enumerative characterization in Theorem~\ref{theo:Nk2exact} tells us more.
Define a poset on the set of line partitions $\P^{3}[n]$ by letting $P \preceq P'$ whenever $P$ is below $P'$ (as in Figure~\ref{fig:1partitionInclusion}). This poset is known in the literature as (the restricted) Young's lattice $L(n,n)$ (see, e.g.,~\cite{Stanley91}, Section 3.1.2), and from Theorem~\ref{theo:Nk2exact} it follows that
$N_4(2,n)-1$ is \emph{exactly} equal to the number of antichains in $L(n,n)$ (as it clearly equals the number of down-sets there).
A well-known result in enumerative combinatorics states that this poset has the Sperner property (this was first shown by Stanley in~\cite{Stanley80}),
implying that the largest antichain in it is obtained by taking every line partition such that the area below it is $\frac12 n^2$.
However, the number $M(n)$ of these line partitions is, as far as we know, not well understood, even asymptotically.
It therefore seems that determining $N_4(2,n)$ is hard; even if $L(n,n)$ is a ``nice'' poset, in the sense that the same phenomenon in Dedekind's problem holds and the number of antichains in $L(n,n)$ is $2^{\Theta(M(n))}$, it seems that in order to determine $N_4(2,n)$ one must first know $M(n)$.

Let us mention here that a lower bound for $M(n)$ is not too hard to obtain. Indeed, letting $X$ be the area below a random member of $L(n,n)$, one can show that the variance of $X$ is $O(n^3)$. By Chebyschev's Inequality, the area below most members of $L(n,n)$ deviates from $\frac12 n^2$ by $O(n^{3/2})$, and so $M(n)\geq\Omega(\size{L(n,n)}/n^{3/2})$ (which improves on the ``naive'' bound $M(n)\geq\size{L(n,n)}/(n^2+1)$).
Note that this implies that $\log_2 N_4(2,n)\geq \Omega\big(\binom{2n}{n}/n^{3/2}\big)$, as $N_4(2,n)\geq 2^{M(n)}$ similarly to the proof of Theorem~\ref{theo:lower}.
It seems reasonable to believe that $M(n)$, and respectively $\log_2 N_4(2,n)$, are not much larger than the aforementioned lower bounds.

\paragraph{Transitive colorings:} As in the rest of the paper, let us assume that the vertices of the complete hypergraphs $K^k_N$ are the integers $1,\ldots,N$. We say that a $q$-coloring of the edges of $K^k_N$ is transitive if the following condition holds; for every $(k+1)$-tuple of vertices $x_1 < x_2, \ldots < x_{k+1}$, if the two edges $\{x_1,\ldots, x_{k}\}$, $\{x_2,\ldots, x_{k+1}\}$ received color $i$, then so did the other $k-1$ edges consisting of $k$ of the vertices $x_1,\ldots, x_{k+1}$. Let $N'_{k}(q,n)$ be the variant of $N_{k}(q,n)$ restricted to transitive colorings. The problem of bounding $N'_{k}(q,n)$ was raised by Eli\'{a}\v{s} and Matou\v{s}ek~\cite{EliasMa11}. We clearly have $N'_{k}(q,n) \leq N_{k}(q,n)$ so the main question is whether $N'_{k}(q,n)$ is a tower of height $k-1$ as is $N_{k}(q,n)$. It is not hard to see that the coloring showing that $N_2(q,n)> n^q$ is transitive, implying that $N'_{2}(q,n)=N_{2}(q,n)$. One can also check that the colori
 ng we use in the proof of Lemma~\ref{lemma:N3lower} is transitive, implying that $N'_{3}(q,n)=N_{3}(q,n)$. One is thus tempted to ask if $N'_{k}(q,n)=N_{k}(q,n)$? As it turns out, the coloring we use to prove Lemma~\ref{lemma:Nklower} is {\em not} transitive. So the question of deciding if $N'_{k}(q,n)=N_{k}(q,n)$ remains an interesting open problem.
It might very well be possible to define a variant of our coloring that will be transitive and give comparable bounds.

\paragraph{A better exponent for $P_d(n)$:} It is not hard to see that one can derive from (\ref{eq:dim2partition}) the bound $N_3(3,n)=(27/16)^{3/2\cdot n^2(1-o(1))} = 2^{(c-o(1))n^2}$, where $c=\frac32(3\log_2 3-4)\approx 1.1323$.
This of course means that the bound in (\ref{eq:dimqpartition}) can be improved to $P_d(n) \leq 2^{(c-o(1))n^{d}}$ (for $d \geq 2$), implying similar improvements of the constants involved in the results stated in Subsections~\ref{subsec:results1} and~\ref{subsection:results2} (for $q \geq 3$).
This can also be used to show that $N_k(3,n)=t_{k-1}((c-o(1))n)$.


\bigskip

\noindent \textbf{Acknowledgment:} We are extremely grateful to Benny Sudakov for suggesting to us some of the problems studied in this paper.

\end{document}